\documentclass[10pt, a4paper, UKenglish]{article}

\usepackage[UKenglish]{babel}

\usepackage{amssymb,amsmath,amsfonts,amsthm,nomencl,mathrsfs } 
\usepackage{bbm}
\usepackage[arrow, matrix, curve]{xy}
\usepackage[latin1]{inputenc}
\usepackage{a4wide}
\usepackage{color}
\usepackage{hyperref}
\usepackage{upref}
\usepackage{dirtytalk}
\usepackage{physics}
\usepackage{enumerate}
\usepackage{shuffle}
\usepackage{slashed}

\newcommand{\IC}{\mathbb{C}}
\newcommand{\IR}{\mathbb{R}}

\newcommand{\question}[1]{\leavevmode{\marginpar{\tiny%
			$\hbox to 0mm{\hspace*{-0.5mm}$\leftarrow$\hss}%
			\vcenter{\vrule depth 0.1mm height 0.1mm width \the\marginparwidth}%
			\hbox to 0mm{\hss$\rightarrow$\hspace*{-0.5mm}}$\\\relax\raggedright #1}}}

\renewcommand{\tr}{\mathrm{tr}}

\newcommand{\dom}{\mathrm{Dom}}

\newcommand{\IN}{\mathbb{N}}
\newcommand{\IZ}{\mathbb{Z}}
\newcommand{\Ii}{\mathbbm{i}}

\newcommand{\Id}{{\rm d}}

\newcommand{\p}{\partial}

\newcommand{\one}{\mathbbm{1}}

\theoremstyle{plain}            
\newtheorem{theorem}{theorem}[section]
\newtheorem{Lemma}[theorem]{Lemma}
\newtheorem{Corollary}[theorem]{Corollary}
\newtheorem{Theorem}[theorem]{Theorem}
\newtheorem{Proposition}[theorem]{Proposition}

\theoremstyle{definition}
\newtheorem{Definition}[theorem]{Definition}

\newtheorem{Hypothesis}[theorem]{Hypothesis}
\newtheorem{Remark}[theorem]{Remark}

\DeclareMathOperator{\ind}{ind}

\title{The Witten Index of massless $(d+1)$-Dirac-Schr\"odinger Operators}
\author{Oliver F\"urst}
\date{24th May 2024}

\begin{document}
	\maketitle
	
	\begin{abstract}
		We calculate the Witten index of a class of (non-Fredholm) Dirac-Schr\"odinger operators over $\IR^{d+1}$ for $d\geq 3$ odd, and thus generalize known results for the case $d=1$. For a concrete example of the potential, we give a more explicit index formula, showing that the Witten index assumes any real number on this class of operators.
	\end{abstract}
	
	\section{Introduction and Main Result}
	
	The (semi-group regularized) partial Witten index of a densely defined operator $T:\dom\ T\subseteq\IC^{r}\otimes H\rightarrow\IC^{r}\otimes H$ for a separable Hilbert space $H$ and $r\in\IN$ is the limit
	\begin{align}
		\mathrm{ind}_{W}\ T:=\lim_{t\to\infty}\tr_{H}\tr_{\IC^{r}}\left(e^{-tT^{\ast}T}-e^{-tTT^{\ast}}\right),
	\end{align}
	whenever the right-hand side exists. In particular, the partial trace $\tr_{\IC^{r}}\left(e^{-tT^{\ast}T}-e^{-tTT^{\ast}}\right)$ needs to be trace-class in $H$ for some $t>0$ (and thus also for any larger $t$). The notion was introduced in originally in \cite{GesSim}, and provides in some cases a finite value although $T$ might be non-Fredholm. However, in cases where both the Fredholm index and the Witten index exist, they coincide. It is well-known (cf. \cite{BolGesGroSchSim}), that $\mathrm{ind}_{W}\ T$ can take any real value. Massless $(1+1)$-Dirac Schr\"odinger operators are the first known class of differential operators for which the Witten index has this property (cf.\cite{BolGesGroSchSim}, \cite{CGLPSZ}, and \cite{CGGLPSZ}).
	In this paper we will generalize the above mentioned result to massless $(d+1)$-Dirac-Schr\"odinger operators for $d\geq 3$ odd, which expands on the proof idea in \cite{BolGesGroSchSim}. Additionally, the calculation of the Witten index relies decisively on a trace formula found in \cite{F2}, which generalizes the classic index theorem by C. Callias \cite{Cal}.
	
	To be more concrete, we consider the differential operator $D_{V}$ over $\IR^{d+1}$, $d\geq 3$ odd, given by
	\begin{align}
		\left(D_{V}f\right)\left(x,y\right):=&i\sum_{j=1}^{d}c^{j}\p_{x^{j}}f\left(x,y\right)+i\p_{y}f\left(x,y\right)+V\left(x,y\right)f\left(x,y\right),\ x\in\IR^{d},\ y\in\IR,\nonumber\\
		f\in& W^{1,2}\left(\IR^{d+1},\IC^{r}\otimes H\right),
	\end{align}
	where $H$ is a complex separable Hilbert space, $\left(c^{j}\right)_{j=1}^{d}$ are Clifford matrices of rank $r$, i.e.
	\begin{align}
		c^{k}c^{l}+c^{l}c^{k}=-2\delta_{kl}\one_{\IC^{r}},\ k,l\in\left\{1,\ldots,d\right\},
	\end{align}
	and the function $V:\IR^{d}\times\IR\rightarrow B_{sa}\left(H\right)$ is bounded, self-adjoint valued, and sufficiently smooth with certain integrability conditions, and Schatten-von Neumann class properties of the first derivatives (for further details see Hypothesis \ref{mainhyp} below).
	
	The following theorem is the principal result of this work:
	
	\begin{Theorem}[Lemma \ref{satisfactionlem}, and Theorem \ref{mainthm}]\label{thethm}
		Let $V\in C_{b}^{3}\left(\IR^{d}\times\IR,B_{sa}\left(H\right)\right)$, where $B_{sa}\left(H\right)$ is equipped with the strong operator topology. Assume furthermore that for $i\in\left\{1,\ldots,d\right\}$ and $k\in\left\{0,1,2\right\}$,
		\begin{align}
			\left\|y\mapsto\p_{x^{i}}\p_{y}^{k}V\left(x,y\right)\right\|_{L^{d}\left(\IR,S^{d}\left(H\right)\right)}&=O\left(\left|x\right|^{-1}\right),\ \left|x\right|\to\infty,\nonumber\\
			\exists\epsilon>0:\ \left\|y\mapsto\p_{R}\p_{y}^{k}V\left(x,y\right)\right\|_{L^{d}\left(\IR,S^{d}\left(H\right)\right)}&=O\left(\left|x\right|^{-1-\epsilon}\right),\ \left|x\right|\to\infty,
		\end{align}
		where $\p_{R}:=\sum_{i=1}^{d}\frac{x^{i}}{\left|x\right|}\p_{x^{i}}$ is the radial derivative, and that for any $\widehat{x}\in S_{1}\left(0\right)$, $x_{0}\in\IR^{d}$, and $\gamma\in\IN_{0}^{d}$ with $\left|\gamma\right|\leq 1$, we have
		\begin{align}
			\lim_{R\to\infty}R^{\left|\gamma\right|}\left\|\left(\p^{\gamma}_{x}V\right)\left(x_{0}+R\widehat{x},\cdot\right)-\left(\p^{\gamma}_{x}V\right)\left(R\widehat{x},\cdot\right)\right\|_{L^{\infty}\left(\IR,B\left(H\right)\right)}=0.
		\end{align}
		Then
		\begin{align}\label{thethmeq1}
			\ind_{W}D_{V}=\frac{1}{2\pi}\left(4\pi\right)^{-\frac{d-1}{2}}\frac{\left(\frac{d-1}{2}\right)!}{d!}\kappa_{c}\int_{\IR^{d}}\tr_{H}\left(\left(U^{V}\right)^{-1}\Id U^{V}\right)^{\wedge d},
		\end{align}
		where $\wedge d$ is the $d$-fold exterior power,
		\begin{align}
			\kappa_{c}:=\tr_{\IC^{r}}\left(c^{1}\cdot\ldots\cdot c^{d}\right),
		\end{align}
		and the unitary $U^{V}$ is given by $U^{V}\left(x\right):=\lim_{n\to\infty}U^{V\left(x,\cdot\right)}\left(n,-n\right)$, $x\in\IR^{d}$, where $U^{A}\left(y_{1},y_{2}\right)$, $y_{1},y_{2}\in\IR$, for a given $A:\IR\rightarrow B_{sa}\left(H\right)$ is the (unique) evolutions system of the equation
		\begin{align}
			u'\left(y\right)=\Ii A\left(y\right) u\left(y\right),\ y\in\IR,
		\end{align}
		i.e. for $y,y_{1},y_{2},y_{3}\in\IR$,
		\begin{align}
			\p_{y_{1}}U^{A}\left(y_{1},y_{2}\right)=&\Ii A\left(y_{1}\right)U^{A}\left(y_{1},y_{2}\right),\\
			\p_{y_{2}}U^{A}\left(y_{1},y_{2}\right)=&-\Ii U^{A}\left(y_{1},y_{2}\right)A\left(y_{2}\right),\\
			U^{A}\left(y_{1},y_{2}\right)U^{A}\left(y_{2},y_{3}\right)=&U^{A}\left(y_{1},y_{3}\right),\\
			U^{A}\left(y,y\right)=&1_{H}.
		\end{align}
		
		With the minimal choice $r=2^{\frac{d-1}{2}}$, and sign convention\footnote{This corresponds to the choice $c^{j}=-\Ii\sigma^{j}$, where $\sigma^{j}$ are classical Pauli matrices.} $\kappa_{c}=\left(-\Ii\right)^{d}\left(2\Ii\right)^{\frac{d-1}{2}}$, we have
		\begin{align}\label{ththmeq2}
			\ind_{W}D_{V}=\left(2\pi\Ii\right)^{-\frac{d+1}{2}}\frac{\left(\frac{d-1}{2}\right)!}{d!}\int_{\IR^{d}}\tr_{H}\left(\left(U^{V}\right)^{-1}\Id U^{V}\right)^{\wedge d}.
		\end{align}
	\end{Theorem}

	As an example, we consider in section \ref{examplesec} a potential of the form $V\left(x,y\right)=\Ii\phi\left(y\right)\sum_{j=1}^{d}\sigma^{j}F_{j}$, $x\in\IR^{d}$, $y\in\IR$, for functions $F:\IR^{d}\rightarrow\IR^{d}$, $\phi\in C_{c}^{\infty}\left(\IR\right)$ with $\int_{\IR}\phi\left(y\right)\Id y=1$, and $\left(\sigma^{j}\right)_{j=1}^{d}$ Clifford matrices (for more details see Definition \ref{potentialexdef}). Note that in $d=3$ this example covers $su\left(2\right)$-valued potentials. The index formula for this class of potentials then simplifies to
	\begin{Proposition}[Proposition \ref{calclem}]
		For $V$ as in Definition \ref{potentialexdef}, we have
		\begin{align}
			\ind_{W}D_{V}=\left(2\pi\right)^{-\frac{d+1}{2}}\frac{\left(\frac{d-1}{2}\right)!}{d}\int_{\IR^{d}}\frac{\left(\left|F\right|+d-1\right)\left(\cos\left(2\left|F\right|\right)-1\right)^{\frac{d-1}{2}}}{\left|F\right|^{d}}\det\mathrm{D}F\ \Id x,
		\end{align}
		where $\mathrm{D}F$ is the Jacobian of $F$.
	\end{Proposition}
	An immediate consequence of the above formula is the invariance under pre-composition with diffeomorphisms of $\IR^{d}$ in the $x$-variable, and the surjectivity of the Witten index on this class of Dirac-Schr\"odinger operators.
	
	\section{Conditions on the potential $V$}
	
	Throughout this paper we assume that $d\geq 3$ is odd. We quote the main Hypothesis of \cite{F2}, and introduce its terminology before we proceed to check that the potential $V$ is amenable.
	
	$A_{0}$ is a self-adjoint operator in a separable complex Hilbert space $X$, and
	\begin{align}
		M_{0}:L^2\left(\IR^{d},\IC^{r}\otimes\dom\ A_{0}\right)\rightarrow L^2\left(\IR^{d},\IC^{r}\otimes X\right)
	\end{align}
	the self-adjoint multiplication operator with $A_{0}$, given by $\left(M_{0}f\right)\left(x\right)=A_{0}f\left(x\right)$. With slight abuse of notation, we also use $M_{0}$ to denote the constant operator family $\left(\IR^{d}\ni x\mapsto A_{0}\right)$. With $\langle x\rangle_{\left(z\right)}$ we denote the Japanese bracket
	\begin{align}
		\langle x\rangle_{z}=\left(\left|x\right|^{2}+z\right)^{-\frac{1}{2}},\ \langle x\rangle=\langle x\rangle_{1},\ z\in\IC\backslash\left(-\infty,0\right].
	\end{align}
	Let $S^{p}\left(X\right)$ denote the $p$-Schatten-von Neumann operators in $X$ for $p\in\left[1,\infty\right)$. Finally $\p_{R}:=\sum_{i=1}^{d}\frac{x^{i}}{\left|x\right|}\p_{x^{i}}$ denotes the radial derivative. 
	
	We have the following notions, encapsulating the regularity and decay conditions of the perturbation $B$, if we set $A=M_{0}+B$.
	
	\begin{Definition}
		For $B=\left(\IR^{d}\ni x\mapsto B\left(x\right)\right)$ a family of operators in $H$ with domain $\dom\ A_{0}$, we write $B\in W^{k,\infty}\left(\IR^{d},A_{0},m\right)$, for $k\in\IN_{0}$, $m\geq 0$, if $\langle A_{0}\rangle^{\beta}B\left(x\right)\langle A_{0}\rangle^{-1-\beta}$, $x\in\IR^{d}$, is a family of closable operators with bounded closures in $H$ for any $\beta\in\left[-m,m\right]$, and $x\mapsto\overline{\langle A_{0}\rangle^{\beta}B\left(x\right)\langle A_{0}\rangle^{-1-\beta}}\in W^{k,\infty}\left(\IR^{d},B\left(X\right)\right)$, where $W^{k,\infty}$ is the $L^{\infty}$-Sobolev space of order $k$, and $B\left(X\right)$ is equipped with the strong operator topology (SOT).
	\end{Definition}
	
	\begin{Definition}
		We introduce the following semi-norms for operators $B\in W^{1,\infty}\left(\IR^{d},A_{0},1\right)$, with $z\in\IC\backslash\left(-\infty,0\right]$,
		\begin{align}\label{rhodefeq1}
			\rho_{z}\left(B\right):=&\left\|\left\|B\langle M_{0}\rangle^{-1}_{z}\right\|_{B\left(X\right)}\right\|_{L^{\infty}\left(\IR^{d}\right)}+\left\|\left\|\langle M_{0}\rangle B\langle M_{0}\rangle^{-1}\langle M_{0}\rangle^{-1}_{z}\right\|_{B\left(X\right)}\right\|_{L^{\infty}\left(\IR^{d}\right)}\nonumber\\
			&+\sum_{i=1}^{d}\left\|\left\|\p_{i}^{E}B\langle M_{0}\rangle^{-1}\langle M_{0}\rangle^{-1}_{z}\right\|_{B\left(X\right)}\right\|_{L^{\infty}\left(\IR^{d}\right)}.
		\end{align}
		For $\beta\in\IR$ denote for $B\in W^{1,\infty}\left(\IR^{d},A_{0},\left|\beta\right|+1\right)$,
		\begin{align}\label{rhodefeq2}
			\rho_{z}^{\beta}\left(B\right):=\rho_{z}\left(\langle M_{0}\rangle^{\beta}B\langle M_{0}\rangle^{-\beta}\right).
		\end{align}
	\end{Definition}
	
	\begin{Definition}
		We introduce a second set of semi-norms for $B\in W^{1,\infty}\left(\IR^{d},A_{0},m\right)$, $m\geq 0$, and the following data: $\alpha\geq 0$, $\beta\in\IR$, with $\left|\beta-1\right|\leq m$, $\left|\beta-\alpha\right|\leq m$, $p\geq 2$, $s\geq 0$, and a unit vector field $\nu\in\Gamma\left(T\IR^{d}\right)$,
		\begin{align}
			\tau^{\alpha,\beta,\nu,p,s}\left(B\right):=\left\|x\mapsto\langle x\rangle^{s}\left\|\langle A_{0}\rangle^{-\alpha+\beta}\left(\nu^{E}B\right)\left(x\right)\langle A_{0}\rangle^{-\beta}\right\|_{S^{p}\left(X\right)}\right\|_{L^{\infty}\left(\IR^{d}\right)},
		\end{align}
		and denote $\tau^{\alpha,\beta,\nabla,p,s}:=\sum_{i=1}^{d}\tau^{\alpha,\beta,\p_{i},p,s}$.
	\end{Definition}
	
	If we want to apply Theorem 2.31 of \cite{F2}, the following Hypothesis needs to hold.
	
	\begin{Hypothesis}[\cite{F2}, Hypothesis 1.18]\label{mainhyp}
		Let $\alpha\geq 1$, and $N\geq\lfloor\frac{\alpha-1}{2}\left(d+1\right)\rfloor+1$. Let $B$ be a family of symmetric operators in $X$, with domains $\dom\ A_{0}$. Assume that $B\langle A_{0}\rangle^{-1}$ is twice continuously differentiable in SOT, and that $B\in W^{2,\infty}\left(\IR^{d},A_{0},2N\right)$. Assume that
		\begin{enumerate}
			\item \begin{align}\label{hypdomeq}
				\lim_{\mathrm{dist}\left(z,\left(-\infty,0\right]\right)\to\infty}\rho_{z}^{-2N}\left(B\right)+\rho_{z}^{2N}\left(B\right)=0,
			\end{align}
			and that for all $\beta\in\left[-2N+\alpha,2N+1\right]$,
			\begin{align}\label{hyptraceeq}
				\tau^{\alpha,\beta,\nabla,d,1}\left(B\right)&<\infty,\nonumber\\
				\exists\epsilon>0:\ \tau^{\alpha,\beta,\p_{R},d,1+\epsilon}\left(B\right)&<\infty.
			\end{align}
			\item Let $S_{1}\left(0\right)$ denote the $\left(d-1\right)$-dimensional unit sphere. For $\phi\in\dom\ A_{0}$, a.e. $y\in S_{1}\left(0\right)$, a.e. $x\in\IR^{d}$, and $\gamma\in\IN_{0}^{d}$ with $\left|\gamma\right|\leq 1$, assume,
			\begin{align}\label{radlimeq}
				\lim_{R\to\infty}R^{\left|\gamma\right|}\left\|\left(\p^{\gamma}B\left(Ry+x\right)-\p^{\gamma}B\left(Ry\right)\right)\phi\right\|_{X}&=0. 
			\end{align}
		\end{enumerate}
	\end{Hypothesis}
	
	We need to check Hypothesis \ref{mainhyp} in the setup of $(d+1)$-Dirac-Schr\"odinger operators with the following identifications:
	\begin{align}\label{identificationeq}
		X=L^2\left(\IR,H\right),\ A_{0}=\Ii\p\otimes\one_{H},\ B=V.
	\end{align}
	
	\begin{Lemma}\label{satisfactionlem}
		Let $V\in C_{b}^{3}\left(\IR^{d+1},B_{sa}\left(H\right)\right)$, where $B_{sa}\left(H\right)$ is equipped with the SOT. Assume furthermore that for $i\in\left\{1,\ldots,d\right\}$ and $k\in\left\{0,1,2\right\}$,
		\begin{align}
			\left\|y\mapsto\p_{x^{i}}\p_{y}^{k}V\left(x,y\right)\right\|_{L^{d}\left(\IR,S^{d}\left(H\right)\right)}&=O\left(\left|x\right|^{-1}\right),\ \left|x\right|\to\infty,\nonumber\\
			\exists\epsilon>0:\ \left\|y\mapsto\p_{R}\p_{y}^{k}V\left(x,y\right)\right\|_{L^{d}\left(\IR,S^{d}\left(H\right)\right)}&=O\left(\left|x\right|^{-1-\epsilon}\right),\ \left|x\right|\to\infty,
		\end{align}
		and that for any $\widehat{x}\in S_{1}\left(0\right)$, $x_{0}\in\IR^{d}$, and $\gamma\in\IN_{0}^{d}$ with $\left|\gamma\right|\leq 1$, we have
		\begin{align}
			\lim_{R\to\infty}R^{\left|\gamma\right|}\left\|\left(\p^{\gamma}_{x}V\right)\left(x_{0}+R\widehat{x},\cdot\right)-\left(\p^{\gamma}_{x}V\right)\left(R\widehat{x},\cdot\right)\right\|_{L^{\infty}\left(\IR,B\left(H\right)\right)}=0.
		\end{align}
		Then Hypothesis \ref{mainhyp} holds for the choice (\ref{identificationeq}) for $\alpha>1$ small enough, such that $N=1$.
	\end{Lemma}
	
	\begin{proof}
		Throughout the proof we will use the coordinates $\left(x,y\right)\in\IR^{d}\times\IR=\IR^{d+1}$. We first need to check that $V\in W^{2,\infty}\left(\IR^{d},\Ii\p\otimes\one,2\right)$, which means that
		\begin{align}\label{satisfactionlemeq1}
			x\mapsto\overline{\langle\Ii\p\otimes\one_{H}\rangle^{\beta}V\left(x,\cdot\right)\langle\Ii\p\otimes\one_{H}\rangle^{-1-\beta}}\in W^{2,\infty}\left(\IR^{d},B\left(L^2\left(\IR,H\right)\right)\right),\ \beta\in\left[-2,2\right].
		\end{align}
		Because $V$ is self-adjoint valued, (\ref{satisfactionlemeq1}) holds by taking adjoints and complex interpolation, if
		\begin{align}
			x\mapsto\left(\Delta\otimes\one_{H}+1\right)V\left(x,\cdot\right)\left(\Delta\otimes\one_{H}+1\right)^{-1}\in W^{2,\infty}\left(\IR^{d},B\left(L^2\left(\IR,H\right)\right)\right),
		\end{align}
		which is satisfied by the Leibniz rule for $V\in C^{2}_{b}\left(\IR^{d+1},B_{sa}\left(H\right)\right)$. Next, for $z\in\IC\backslash\left(-\infty,0\right]$, we need to verify
		\begin{align}\label{satisfactionlemeq2}
			\lim_{\mathrm{dist}\left(z,\left(-\infty,0\right]\right)\to\infty}\rho_{z}^{-2}\left(V\right)+\rho_{z}^{2}\left(V\right)=0.
		\end{align}
		The operator functions $\langle M_{0}\rangle^{\beta}V\langle M_{0}\rangle^{-\beta-1}$, $\langle M_{0}\rangle^{\beta+1}V\langle M_{0}\rangle^{-\beta-2}$, and $\langle M_{0}\rangle^{\beta}\p_{x^{i}}V\langle M_{0}\rangle^{-\beta-1}$ for $i\in\left\{1,\ldots,d\right\}$ are in $L^{\infty}\left(\IR^{d},B\left(L^2\left(\IR,H\right)\right)\right)$ for $\beta\in\left[-2,2\right]$, because $V\in C_{b}^{3}\left(\IR^{d+1},B_{sa}\left(H\right)\right)$ by same reasoning as in first part of the proof. Spectral mapping implies that $\left\|\langle M_{0}\rangle_{z}\right\|_{L^{\infty}\left(\IR^{d},B\left(H\right)\right)}\to 0$, as $\mathrm{dist}\left(z,\left(-\infty,0\right]\right)\to\infty$, multiplying out this factor in (\ref{rhodefeq1}) and (\ref{rhodefeq2}), we obtain the claimed property (\ref{satisfactionlemeq2}).
		
		Now let $\nu$ be a unit vector field in $\IR^{d}$. The property $\tau^{\alpha,\beta,\nu,d,s}\left(V\right)<\infty$, for $\beta\in\left[-2+\alpha,3\right]$, and $s\geq 0$, means that
		\begin{align}\label{satisfactionlemeq3}
			x\mapsto\langle x\rangle^{s}\left\|\langle\Ii\p\otimes\one_{H}\rangle^{-\alpha+\beta}\nu_{x}V\left(x,\cdot\right)\langle\Ii\p\otimes\one_{H}\rangle^{-\beta}\right\|_{S^{d}\left(L^2\left(\IR,H\right)\right)}\in L^{\infty}\left(\IR^{d}\right).
		\end{align}
		Again, by taking adjoints and complex interpolation, it suffices to check (\ref{satisfactionlemeq3}) for $\beta=3$ and $\beta=\alpha$, which in turn is satisfied if (\ref{satisfactionlemeq3}) holds for $\beta=2+\alpha$ and $\beta=\alpha$, since $\alpha>1$. We first discuss $\beta=\alpha$. In this case we have to consider
		\begin{align}
			\left\|\nu_{x}V\left(x,\cdot\right)\langle\Ii\p\otimes\one_{H}\rangle^{-\alpha}\right\|_{S^{d}\left(L^2\left(\IR,H\right)\right)}.
		\end{align}
		Now we treat $\beta=2+\alpha$. We have to consider
		\begin{align}
			\left(\left(\Delta\otimes\one_{H}+1\right)\nu_{x}V\left(x,\cdot\right)\left(\Delta\otimes\one_{H}+1\right)^{-1}\right)\langle\Ii\p\otimes\one_{H}\rangle^{-\alpha}
		\end{align}
		By using product rule, and since $\left(\Delta\otimes\one_{H}+1\right)^{-t}$ is norm-bounded by $1$ for $t\geq 0$, we have some global constant $C$ such that
		\begin{align}
			&\left\|\left(\left(\Delta\otimes\one_{H}+1\right)\nu_{x}V\left(x,\cdot\right)\left(\Delta\otimes\one_{H}+1\right)^{-1}\right)\langle\Ii\p\otimes\one_{H}\rangle^{-\alpha}\right\|_{S^{d}\left(L^2\left(\IR,H\right)\right)}\nonumber\\
			\leq& C\sum_{k=0}^{2}\left\|\nu_{x}\p_{y}^{k}V\left(x,\cdot\right)\langle\Ii\p\otimes\one_{H}\rangle^{-\alpha}\right\|_{S^{d}\left(L^2\left(\IR,H\right)\right)}.
		\end{align}
		We note that also the case $\beta=\alpha$ is bounded by the right hand side estimate. Since $d\geq 3$, we may apply Lemma 1.15 in \cite{F2}, which is a slight generalization of Theorem 4.1 in \cite{Simon}. It follows that
		\begin{align}
			\left\|\nu_{x}\p_{y}^{k}V\left(x,\cdot\right)\langle\Ii\p\otimes\one_{H}\rangle^{-\alpha}\right\|_{S^{d}\left(L^2\left(\IR,H\right)\right)}\leq\left(2\pi\right)^{-\frac{1}{d}}\left\|\nu_{x}\p_{y}^{k}V\left(x,\cdot\right)\right\|_{L^{d}\left(\IR,S^{d}\left(H\right)\right)}\left\|\xi\mapsto\langle\xi\rangle^{-\alpha}\right\|_{L^{d}\left(\IR\right)}.
		\end{align}
		Since $\left\|\xi\mapsto\langle\xi\rangle^{-\alpha}\right\|_{L^{d}\left(\IR\right)}$ is finite, we obtain
		\begin{align}
			\tau^{\alpha,\beta,\nabla,d,1}\left(V\right)&<\infty,\nonumber\\
			\exists\epsilon>0:\ \tau^{\alpha,\beta,\p_{R},d,1+\epsilon}\left(V\right)&<\infty,
		\end{align}
		by the prerequisites on $V$.
		
		It remains to check the radial limit property, which follows immediately from estimating
		\begin{align}
			&\left\|\left(\p_{x}^{\gamma}V\left(x_{0}+R\widehat{x},\cdot\right)-\p_{x}^{\gamma}V\left(x_{0}+R\widehat{x},\cdot\right)\right)\phi\right\|_{L^2\left(\IR,H\right)}\nonumber\\
			\leq&\left\|\left(\p^{\gamma}_{x}V\right)\left(x_{0}+R\widehat{x},\cdot\right)-\left(\p^{\gamma}_{x}V\right)\left(R\widehat{x},\cdot\right)\right\|_{L^{\infty}\left(\IR,B\left(H\right)\right)}\left\|\phi\right\|_{L^2\left(\IR,H\right)},
		\end{align}
		for any $\phi\in W^{1,2}\left(\IR,H\right)$, $x_{0}\in\IR^{d}$, $\widehat{x}\in S_{1}\left(0\right)$, and $R>0$.
	\end{proof}
	
	\section{Proof of the Main Theorem}

	To apply the trace formula in \cite{F2} to the operator $D_{V}$, we need to establish the integral kernel of the semi-group generated by $\left(\Ii\p\otimes\one_{H}+V\right)^{2}$ in $L^2\left(\IR,H\right)$. To do so we use the theory of evolution systems. A treaty of evolution systems, including the properties in the Definition below, is for example displayed in chapter 4 of \cite{Paz}.
	
	\begin{Definition}
		Let $H$ be a separable complex Hilbert space. Let $\IR\ni x\mapsto A\left(x\right)\in B_{sa}\left(H\right)$ a continuous function in the SOT of $H$. Denote by $U^{A}$ the (unitary) evolution system of the $H$-valued equation
		\begin{align}
			u'\left(x\right)=\Ii A\left(x\right)u\left(x\right),\ x\in\IR,
		\end{align}
		i.e. for $x,x_{1},x_{2},x_{3}\in\IR$,
		\begin{align}\label{evolutiondefeq1}
			\p_{x_{1}}U^{A}\left(x_{1},x_{2}\right)=\Ii A\left(x_{1}\right)U^{A}\left(x_{1},x_{2}\right),\nonumber\\
			\p_{x_{2}}U^{A}\left(x_{1},x_{2}\right)=-\Ii U^{A}\left(x_{1},x_{2}\right)A\left(x_{2}\right),\nonumber\\
			U^{A}\left(x_{1},x_{2}\right)U^{A}\left(x_{2},x_{3}\right)=U^{A}\left(x_{1},x_{3}\right),\nonumber\\
			U^{A}\left(x,x\right)=1_{H}.
		\end{align}
	\end{Definition}
	
	We also need to regularly extend a evolutions system $U^{A}\left(x_{1},x_{2}\right)$ to the boundary of the real line, i.e. we need that $x_{1},x_{2}\in\IR\cup\left\{\pm\infty\right\}$ is permitted.
	
	\begin{Lemma}\label{extendlem}
		Let $\IR\ni x\mapsto A\left(x\right)\in B_{sa}\left(H\right)$ a continuous function in the SOT of $H$, which additionally satisfies $A\in L^{1}\left(\IR,B_{sa}\left(H\right)\right)$. Then $U^{A}$ extends in SOT of $H$ to $\left(\IR\cup\left\{\pm\infty\right\}\right)^{2}$. Moreover the evolution system properties (\ref{evolutiondefeq1}) also extend in SOT to $x,x_{1},x_{2},x_{3}\in\IR\cup\left\{\pm\infty\right\}$.
	\end{Lemma}
	
	\begin{proof}
		We first introduce the limits
		\begin{align}\label{extendlemeq0}
			U^{A}\left(x,-\infty\right)=&\lim_{n\to\infty}U^{A}\left(x,-n\right),\ x\in\IR,\nonumber\\
			U^{A}\left(+\infty,x\right)=&\lim_{n\to\infty}U^{A}\left(n,x\right),\ x\in\IR,\nonumber\\
			U^{A}\left(+\infty,-\infty\right)=&\lim_{n\to\infty}U^{A}\left(n,-n\right)=\lim_{n\to\infty}U^{A}\left(+\infty,-n\right)=\lim_{n\to\infty}U^{A}\left(n,-\infty\right),
		\end{align}
		which exist in the SOT of $H$, since for $\phi\in H$,
		\begin{align}\label{extendlemeq1}
			&\left\|U^{A}\left(x_{1},x_{2}\right)\phi-\phi\right\|_{H}=\left\|\int_{x_{2}}^{x_{1}}\Ii A\left(w\right)U^{A}\left(w,x_{2}\right)\phi\Id w\right\|_{H}\leq\int_{x_{2}}^{x_{1}}\left\|A\left(w\right)\right\|_{B\left(H\right)} \left\|\phi\right\|_{H}\Id w\nonumber\\
			&\xrightarrow{x_{1},x_{2}\to+\infty,\ \text{or}\ x_{1},x_{2}\to -\infty}0,
		\end{align}
		which in turn, by the evolution properties of $U^{A}$, implies that the sequences in (\ref{extendlemeq0}) are Cauchy sequences in the SOT of $H$. Moreover we note that the evolution properties
		\begin{align}
			U^{A}\left(x_{1},x_{3}\right)&=U^{A}\left(x_{1},x_{2}\right)U^{A}\left(x_{2},x_{3}\right),\nonumber\\
			U^{A}\left(x,x\right)&=1_{H},
		\end{align}
		extend to $x,x_{1},x_{2},x_{3}\in\IR\cup\left\{\pm\infty\right\}$. For differentiation in SOT we note for $x\in\IR$,
		\begin{align}
			&\lim_{h\to 0}h^{-1}\left(U^{A}\left(x+h,-\infty\right)-U^{A}\left(x,-\infty\right)\right)\nonumber\\
			=&\lim_{h\to 0}h^{-1}\left(U^{A}\left(x+h,x\right)-1\right)U^{A}\left(x,-\infty\right)=\Ii A\left(x\right)U^{A}\left(x,-\infty\right).
		\end{align}
		The analogous differentiability result also holds for $U^{A}\left(+\infty,x\right)$.
	\end{proof}
	
	\begin{Lemma}\label{homotopielem}
		Let $\IR^{d}\times\IR\ni\left(x,y\right)\mapsto V\left(x,y\right)\in B_{sa}\left(H\right)$ be a function continuous in the $y$-variable in the SOT of $H$ and continuously differentiable in the $x$-variable in the SOT of $H$, which additionally satisfies $V\left(x,\cdot\right),\ \p_{x^{i}}V\left(x,\cdot\right)\in L^{1}\left(\IR,B_{sa}\left(H\right)\right)$, for all $x\in\IR^{d}$, $i\in\left\{1,\ldots,d\right\}$. Then
		\begin{align}
			\IR^{d}\ni x\mapsto U^{V}\left(x\right):=U^{V\left(x,\cdot\right)}\left(+\infty,-\infty\right),
	\end{align}
	is continuously differentiable in the SOT of $H$ with
	\begin{align}\label{homotopielemeq0}
		\p_{x^{i}}U^{V}\left(x\right)=\int_{\IR} U^{V\left(x,\cdot\right)}\left(\infty,y\right)\Ii\p_{x^{i}}V\left(x,y\right)U^{V\left(x,\cdot\right)}\left(y,-\infty\right)\Id y.
	\end{align}
	\end{Lemma}
	
	\begin{proof}
		For $x_{1},x_{2}\in\IR^{d}$, we have for $y,z\in\IR\cup\left\{\pm\infty\right\}$,
		\begin{align}\label{homotopielemeq1}
			&U^{V\left(x_{1},\cdot\right)}\left(y,z\right)-U^{V\left(x_{2},\cdot\right)}\left(y,z\right)=\left[U^{V\left(x_{2},\cdot\right)}\left(y,w\right)U^{V\left(x_{1},\cdot\right)}\left(w,z\right)\right]_{w=z}^{w=y}\nonumber\\
			=&\Ii\int_{z}^{y}U^{V\left(x_{2},\cdot\right)}\left(y,z\right)\left(V\left(x_{1},w\right)-V\left(x_{2},w\right)\right)U^{V\left(x_{1},\cdot\right)}\left(w,z\right)\Id w\xrightarrow{x_{2}\to x_{1}}0,
		\end{align}
		which holds in the SOT of $H$, since $\left\|U^{V\left(x,\cdot\right)}\left(y,z\right)\right\|_{B\left(H\right)}=1$, and the continuity of $V$ in the first variable. Similarly for $\left(e_{i}\right)_{i=1}^{d}$ an orthonormal basis of $\IR^{d}$, we have	
		\begin{align}\label{homotopielemeq2}
			\p_{x^{i}}U^{V}\left(x\right)&=\lim_{h\to 0}h^{-1}\left(U^{V}\left(x+he_{i}\right)-U^{V}\left(x\right)\right)\nonumber\\
			&=\lim_{h\to 0}h^{-1}\left[U^{V\left(x,\cdot\right)}\left(\infty,y\right)U^{V\left(x+he_{i},\cdot\right)}\left(y,-\infty\right)\right]^{y=+\infty}_{y=-\infty}\nonumber\\
			&=\lim_{h\to 0}\int_{\IR}U^{V\left(x,\cdot\right)}\left(\infty,y\right)\Ii\frac{V\left(x+he_{i},y\right)-V\left(x,y\right)}{h}U^{V\left(x+he_{i},\cdot\right)}\left(y,-\infty\right)\Id y\nonumber\\
			&=\int_{\IR} U^{V\left(x,\cdot\right)}\left(\infty,y\right)\Ii\p_{x^{i}}V\left(x,y\right)U^{V\left(x,\cdot\right)}\left(y,-\infty\right)\Id y,
		\end{align}
		where the limits are in the SOT of $H$, and we used the SOT-continuity of $x\mapsto U^{V\left(x,\cdot\right)}\left(y,z\right)$ shown in (\ref{homotopielemeq1}). Combining both partial differentiability and continuity, it follows that $U^{V}$ is continuously differentiable in the SOT of $H$.
	\end{proof}
	
	\begin{Lemma}\label{kernellem}
		Let $t>0$, and $V:\IR\rightarrow B\left(H\right)$ a SOT-differentiable function  Then $H_{t}^{V}\left(x,y\right)$ given by
		\begin{align}
			H_{t}^{V}\left(x,y\right)&:=q_{t}\left(x-y\right)U^{V}\left(x,y\right),\nonumber\\
			q_{t}\left(z\right)&:=\left(4\pi t\right)^{-\frac{1}{2}}e^{-\frac{\left(x-y\right)^{2}}{4t}},
		\end{align}
		is the integral kernel of $e^{-t\left(\Ii\p\otimes 1_{H}+ V\right)^{2}}$ in $L^{2}\left(\IR\right)\widehat{\otimes}H=L^2\left(\IR,H\right)$, which is differentiable in the SOT.
	\end{Lemma}
	
	\begin{proof}
		We use that $q_{t}$ is the smooth heat kernel of $e^{t\p^{2}}$ in $L^2\left(\IR\right)$. We calculate
		\begin{align}
			&\p_{t}H_{t}^{V}\left(x,y\right)=q_{t}''\left(x-y\right)U^{V}\left(x,y\right),\\
			&\p_{x}H_{t}^{V}\left(x,y\right)=-\frac{x-y}{2t}H_{t}^{V}\left(x,y\right)+\Ii V\left(x\right)H_{t}^{V}\left(x,y\right),\\
			&\p_{x}^{2}H_{t}^{V}\left(x,y\right)=q_{t}''\left(x-y\right)U^{V}\left(x,y\right)-\Ii\frac{x-y}{t}V\left(x\right)H_{t}^{V}\left(x,y\right)+\Ii V'\left(x\right)H^{V}_{t}\left(x,y\right)-V\left(x\right)^{2}H^{V}_{t}\left(x,y\right),\\
			&\left(\Ii\p_{x}+V\left(x\right)\right)^{2}H_{t}^{V}\left(x,y\right)=-q_{t}''\left(x-y\right)U^{V}\left(x,y\right)=-\p_{t}H_{t}^{V}\left(x,y\right).
		\end{align}
		We also have for $f\in C_{c}^{\infty}\left(\IR,H\right)$,
		\begin{align}
			\int_{\IR}H_{t}^{V}\left(x,y\right)f\left(y\right)\Id y=\int_{\IR}q_{1}\left(w\right)U^{V}\left(x,x-\sqrt{t}w\right)f\left(x-\sqrt{t}w\right)\Id w\xrightarrow{t\searrow 0}U^{V}\left(x,x\right)f\left(x\right)=f\left(x\right),
		\end{align}
		which shows that $H^{V}_{t}$ is the claimed integral kernel.
	\end{proof}

	\begin{Theorem}\label{mainthm}
		Let $V$ satisfy the prerequisites of Lemma \ref{satisfactionlem}. Then
		\begin{align}\label{mainthmeq5}
			\ind_{W}D=\frac{1}{2\pi}\left(4\pi\right)^{-\frac{d-1}{2}}\frac{\left(\frac{d-1}{2}\right)!}{d!}\kappa_{c}\int_{\IR^{d}}\tr_{H}\left(\left(U^{V}\right)^{-1}\Id U^{V}\right)^{\wedge d},
		\end{align}
		where $U^{V}\left(x\right):=U^{V\left(x,\cdot\right)}\left(+\infty,-\infty\right)$. With the minimal choice $r=2^{\frac{d-1}{2}}$, and sign convention $\kappa_{c}=\left(-\Ii\right)^{d}\left(2\Ii\right)^{\frac{d-1}{2}}$, we have
		\begin{align}\label{mainthmeq6}
			\ind_{W}D=\left(2\pi\Ii\right)^{-\frac{d+1}{2}}\frac{\left(\frac{d-1}{2}\right)!}{d!}\int_{\IR^{d}}\tr_{H}\left(\left(U^{V}\right)^{-1}\Id U^{V}\right)^{\wedge d}.
		\end{align}
	\end{Theorem}

	\begin{proof}
		We apply Theorem 2.31 of \cite{F2} with the choice $\phi=0$, and have
		\begin{align}\label{mainthmeq0}
			&\tr_{L^{2}\left(\IR^{d},L^{2}\left(\IR,H\right)\right)}\tr_{\IC^{r}}\left(e^{-tD^{\ast}D}-e^{-tDD^{\ast}}\right)\nonumber\\
			=&\frac{2}{d}\left(4\pi\right)^{-\frac{d}{2}}t^{\frac{d}{2}}\Ii^{d}\kappa_{c}\int_{\IR^{d}}\int_{u\in\Delta_{d-1}}\tr_{L^{2}\left(\IR,H\right)}\left(\Id A\wedge e^{-tuA^2}\shuffle_{\wedge}\left(\Id A\right)^{\otimes\left(d-1\right)}\right)\Id u,
		\end{align}
		where $A\left(x\right)=\Ii\p+V\left(x,\cdot\right)$. Consider the integrand,
		\begin{align}
			X_{t}:=t^{\frac{d}{2}}\int_{u\in\Delta_{d-1}}\tr_{L^{2}\left(\IR,H\right)}\left(\Id A\wedge e^{-tuA^2}\shuffle_{\wedge}\left(\Id A\right)^{\otimes\left(d-1\right)}\right)\Id u.
		\end{align}
		By Lemma \ref{kernellem}, we have with convention $z_{0}:=z_{d}:=y$,
		\begin{align}
			X_{t}&=t^{\frac{d}{2}}\int_{u\in\Delta_{d-1}}\int_{y\in\IR}\int_{z\in\IR^{d-1}}\tr_{H}\left(\bigwedge_{j=1}^{d}\left(\Id_{x}V\right)\left(x,z_{j-1}\right)H^{V\left(x,\cdot\right)}_{tu_{j}}\left(z_{j-1},z_{j}\right)\right)\Id z\ \Id y\ \Id u\nonumber\\
			&=\int_{u\in\Delta_{d-1}}\int_{z\in\IR^{d}}\prod_{j=1}^{d}q_{u_{j}}\left(\frac{z_{j-1}-z_{j}}{t}\right)\tr_{H}\left(\bigwedge_{j=1}^{d}\left(\Id_{x}V\right)\left(x,z_{j-1}\right)U^{V\left(x,\cdot\right)}\left(z_{j-1},z_{j}\right)\right)\Id z\ \Id u
		\end{align}
		The factors $q_{u_{j}}\left(\frac{z_{j-1}-z_{j}}{t}\right)$ has the integrable dominant $\left(4\pi u_{j}\right)^{-\frac{1}{2}}$, which enables us to take the limit $t\to\infty$ beyond the integral. Thus, with
		\begin{align}
			\int_{u\in\Delta_{d-1}}\prod_{j=0}^{d-1}\left(4\pi u_{j}\right)^{-\frac{1}{2}}\Id u=\left(4\pi\right)^{-\frac{1}{2}}\frac{\left(\frac{d-1}{2}\right)!}{\left(d-1\right)!},
		\end{align}
		we get,
		\begin{align}\label{mainthmeq2}
			\lim_{t\to\infty}X_{t}&=\left(4\pi\right)^{-\frac{1}{2}}\frac{\left(\frac{d-1}{2}\right)!}{\left(d-1\right)!}\int_{z\in\IR^{d}}\tr_{H}\left(\bigwedge_{j=1}^{d}\left(\Id_{x}V\right)\left(x,z_{j-1}\right)U^{V\left(x,\cdot\right)}\left(z_{j-1},z_{j}\right)\right)\Id z.
		\end{align}
		We obtain, using formula (\ref{homotopielemeq0}) of Lemma \ref{homotopielem}, and the evolution properties,
		\begin{align}\label{mainthmeq4}
			&\int_{z\in\IR^{d}}\tr_{H}\left(\bigwedge_{j=1}^{d}\left(\Id_{x}V\right)\left(x,z_{j-1}\right)U^{V\left(x,\cdot\right)}\left(z_{j-1},z_{j}\right)\right)\Id z\nonumber\\
			=&\int_{z\in\IR^{d}}\tr_{H}\left(\bigwedge_{j=1}^{d}U^{V\left(x,\cdot\right)}\left(z_{j-1},-\infty\right)\left(U^{V}\left(x\right)\right)^{-1}U^{V\left(x,\cdot\right)}\left(+\infty,z_{j-1}\right)\left(\Id_{x}V\right)\left(x,z_{j-1}\right)\right.\nonumber\\
			&\quad\quad\left.U^{V\left(x,\cdot\right)}\left(z_{j-1},-\infty\right)U^{V\left(x,\cdot\right)}\left(-\infty,z_{j}\right)\right)\Id z\nonumber\\
			=&\Ii^{-d}\tr_{H}\left(\left(U^{V}\left(x\right)\right)^{-1}\left(\Id U^{V}\right)\left(x\right)\right)^{\wedge d}.
		\end{align}
		Note that in the last step we also commuted the factor $U^{V\left(x,\cdot\right)}\left(-\infty,z_{d}\right)=U^{V\left(x,\cdot\right)}\left(-\infty,z_{0}\right)$ under the trace. We insert (\ref{mainthmeq4}) into (\ref{mainthmeq2}), which we in turn insert into (\ref{mainthmeq0}), and conclude
		\begin{align}
			\lim_{t\to\infty}\tr_{L^{2}\left(\IR^{d},L^{2}\left(\IR,H\right)\right)}\tr_{\IC^{r}}\left(e^{-tD^{\ast}D}-e^{-tDD^{\ast}}\right)=\frac{1}{2\pi}\left(4\pi\right)^{-\frac{d-1}{2}}\frac{\left(\frac{d-1}{2}\right)!}{d!}\kappa_{c}\int_{\IR^{d}}\tr_{H}\left(\left(U^{V}\right)^{-1}\Id U^{V}\right)^{\wedge d}.
		\end{align}
		With the choice $r=2^{\frac{d-1}{2}}$, and $\kappa_{c}=\Ii^{-d}\left(2\Ii\right)^{\frac{d-1}{2}}$, we have in this case
		\begin{align}
			\lim_{t\to\infty}\tr_{L^{2}\left(\IR^{d},L^{2}\left(\IR,H\right)\right)}\tr_{\IC^{r}}\left(e^{-tD^{\ast}D}-e^{-tDD^{\ast}}\right)=\left(2\pi\Ii\right)^{-\frac{d+1}{2}}\frac{\left(\frac{d-1}{2}\right)!}{d!}\int_{\IR^{d}}\tr_{H}\left(\left(U^{V}\right)^{-1}\Id U^{V}\right)^{\wedge d}.
		\end{align}
	\end{proof}
	
	The next simple Lemma recalls what form the integrand of formula (\ref{mainthmeq5}), takes for commutative potentials in the $y$-direction. It follows from the fact that for pointwise commutative families $\IR\ni y\mapsto A\left(y\right)$, the generated evolution system is the exponential $U^{A}\left(y_{1},y_{2}\right)=e^{\Ii\int_{y_{2}}^{y_{1}}A\left(u\right)\Id u}$, $y_{1},y_{2}\in\IR$, as well as the application of a well-known differentiation formula for exponentials of operators.
	
	\begin{Lemma}\label{commlem}
		Let $V:\IR^{d}\times\IR\rightarrow B_{sa}\left(H\right)$ be continuous in SOT of $H$. Assume that for all $x\in\IR^{d}$, and $y_{1},y_{2}\in\IR$ the operators $V\left(x,y_{1}\right)$ and $V\left(x,y_{2}\right)$ commute. Let $v\left(x\right):=\int_{\IR}V\left(x,y\right)\Id y$ for $x\in\IR^{d}$. Then $U^{V}=e^{iv}$, and
		\begin{align}
			\left(U^{V}\right)^{-1}\Id U^{V}=\Ii\int_{0}^{1}e^{-\Ii sv}\ \Id v\  e^{\Ii sv}\Id s.
		\end{align}
	\end{Lemma}

	We now will construct a concrete class of potentials $V$ in $\IR^{d+1}$, for which we can evaluate the right hand side of (\ref{mainthmeq5}), and show that the Witten index of the associated Callias operator $D_{V}$ may assume any real number for any odd $d\geq 3$.
	
	\section{Example}\label{examplesec}
	
	It is apparent that a non-trivial example of index formula of Theorem \ref{mainthm}, requires $V$ to be sufficiently non-commutative in the $x$-variable, due to the presence of the exterior product. An obvious choice would therefore be to consider $su\left(n\right)$-valued potentials $V$ for $n$ large enough. However the general commutator relations in $su\left(n\right)$ are too complicated to calculate in general, so we consider as a subclass those potentials, for which its values are a (real) linear combination of Clifford-matrices. In the special case $d=3$ these are however all $su\left(2\right)$-matrices.
	
	\begin{Definition}\label{potentialexdef}
		Let $\left(\sigma^{j}\right)_{j=1}^{d}$ be Clifford matrices of rank $r'$. Let $F:\IR^{d}\rightarrow\IR^{d}$ be a bounded smooth function with bounded derivatives, such that
		\begin{align}
			\p_{x^{i}}F\left(x\right)&=O\left(\left|x\right|^{-1}\right),\ \left|x\right|\to\infty,\nonumber\\
			\exists\epsilon>0:\ \p_{R}F\left(x\right)&=O\left(\left|x\right|^{-1-\epsilon}\right),\ \left|x\right|\to\infty,
		\end{align}
		and that for any $\widehat{x}\in S_{1}\left(0\right)$, $x_{0}\in\IR^{d}$, and $\gamma\in\IN_{0}^{d}$ with $\left|\gamma\right|\leq 1$,
		\begin{align}
			\lim_{R\to\infty}R^{\left|\gamma\right|}\left|\left(\p^{\gamma}F\right)\left(x_{0}+R\widehat{x}\right)-\left(\p^{\gamma}F\right)\left(R\widehat{x}\right)\right|=0.
		\end{align}
		Denote by $\omega:\IR^{d}\rightarrow\IC^{r'\times r'}$ the (hermitian-valued) function
		\begin{align}
			\omega:=\Ii\sum_{j=1}^{d}\sigma^{j}F_{j},
		\end{align}
		and if $\left|F\right|>0$,
		\begin{align}
			\widehat{\omega}:=\left|F\right|^{-1}\omega,\ \eta:=\sum_{j=1}^{d}\frac{F_{j}}{\left|F\right|}\Id F_{j}.
		\end{align}
		Set $V\left(x,y\right):=\omega\left(x\right)\phi\left(y\right)$, $x\in\IR^{d}$, $y\in\IR$, for some function $\phi\in C_{c}^{\infty}\left(\IR\right)$ with $\int_{\IR}\phi\left(y\right)\Id y=1$.
	\end{Definition}
	
	\begin{Remark}
		For the above choice of $V$, the induced Dirac-Schr\"odinger operator $D_{V}$ is \textit{not} Fredholm. Indeed, although $D_{V}$ is elliptic, the full symbol of $D_{V}$ is not (uniformly) invertible away from $\left(x,y\right)=0$, which is necessary for Fredholmness, according to Theorem 1 in \cite{Cal}.
	\end{Remark}
	
	For this choice of $V$, the index formula can be evaluated more concretely.

	\begin{Proposition}\label{calclem}
		Let $V$ be as in Definition \ref{potentialexdef}. Then
		\begin{align}
			\ind_{W}D_{V}=\left(2\pi\right)^{-\frac{d+1}{2}}\frac{\left(\frac{d-1}{2}\right)!}{d}\int_{\IR^{d}}\frac{\left(\left|F\right|+d-1\right)\left(\cos\left(2\left|F\right|\right)-1\right)^{\frac{d-1}{2}}}{\left|F\right|^{d}}\det\mathrm{D}F\ \Id x,
		\end{align}
		where $\mathrm{D}F$ is the Jacobian of $F$.
	\end{Proposition}
	
	\begin{proof}
		By Lemma \ref{commlem}, we have
		\begin{align}
			\left(U^{V}\right)^{-1}\Id U^{V}=\Ii\int_{0}^{1}e^{-\Ii s\omega}\ \Id\omega\  e^{\Ii s\omega}\Id s.
	\end{align}
	Because of the Clifford anti-commutator relations, we have
	\begin{align}
		e^{\Ii s\omega}=\cos\left(s\left|F\right|\right)+\Ii\sin\left(s\left|F\right|\right)\widehat{\omega},\ s\in\IR,
	\end{align}
	and thus,
	\begin{align}
		\left(U^{V}\right)^{-1}\Id U^{V}=&\int_{0}^{1}\left(\Ii\cos\left(2s\left|F\right|\right)\Id\omega-\sin\left(2s\left|F\right|\right)\left(\eta-\widehat{\omega}\Id\omega\right)+\Ii\left(1-\cos\left(2s\left|F\right|\right)\right)\widehat{\omega}\eta\right)\Id s\nonumber\\
		=& A\Id\omega-B\eta+B\widehat{\omega}\Id\omega+\left(\Ii-A\right)\widehat{\omega}\eta,
	\end{align}
	where
	\begin{align}
		A&:=\Ii\frac{\sin\left(2\left|F\right|\right)}{2\left|F\right|},\nonumber\\
		B&:=\frac{1-\cos\left(2\left|F\right|\right)}{2\left|F\right|}.
	\end{align}
	Taking the square we obtain
	\begin{align}
		\left(\left(U^{V}\right)^{-1}\Id U^{V}\right)^{\wedge 2}=-\left|F\right|^{-1}B\Id\omega\wedge\Id\omega+2\Ii B\eta\wedge\Id\omega+\left(2\left|F\right|^{-1}B+2\Ii A\right)\widehat{\omega}\eta\wedge\Id\omega,
	\end{align}
	and thus for $k\in\IN$, one inductively shows,
	\begin{align}
		\left(\left(U^{V}\right)^{-1}\Id U^{V}\right)^{\wedge 2k}=&\left(-\left|F\right|^{-1}B\right)^{k}\Id\omega^{\wedge 2k}-2\Ii k\left|F\right|\left(-\left|F\right|^{-1}B\right)^{k}\eta\wedge\Id\omega^{\wedge\left(2k-1\right)}\nonumber\\
		&+\left(-2k\left(-\left|F\right|^{-1}B\right)^{k}+2\Ii kA\left(-\left|F\right|^{-1}B\right)^{k-1}\right)\widehat{\omega}\eta\wedge\Id\omega^{\wedge\left(2k-1\right)}.
	\end{align}
	Since $d$ is odd we conclude
	\begin{align}\label{lemexeq1}
		&\left(\left(U^{V}\right)^{-1}\Id U^{V}\right)^{\wedge d}=A\left(-\left|F\right|^{-1}B\right)^{\frac{d-1}{2}}\Id\omega^{\wedge d}+\left(-B+\left(d-1\right)\Ii A\right)\left(-\left|F\right|^{-1}B\right)^{\frac{d-1}{2}}\eta\wedge\Id\omega^{\wedge\left(d-1\right)}\nonumber\\
		&+B\left(-\left|F\right|^{-1}B\right)^{\frac{d-1}{2}}\widehat{\omega}\Id\omega^{\wedge d}\nonumber\\
		&+\left(-\left(d-1\right)A\left(-\left|F\right|^{-1}B\right)^{\frac{d-1}{2}}+\Ii\left(d-1\right)A^2\left(-\left|F\right|^{-1}B\right)^{\frac{d-3}{2}}+\Ii\left(d-1\right)B\left|F\right|\left(-\left|F\right|^{-1}B\right)^{\frac{d-1}{2}}\right.\nonumber\\
		&\left.+\left(\Ii-A\right)\left(-\left|F\right|^{-1}B\right)^{\frac{d-1}{2}}\right)\widehat{\omega}\eta\wedge\Id\omega^{\wedge\left(d-1\right)}.
	\end{align}
	We want to apply the trace $\tr_{\IC^{r'}}$ to (\ref{lemexeq1}). Since $\widehat{\omega}\Id\omega+\Id\omega\widehat{\omega}=2\eta$, we have
	\begin{align}
		\tr_{\IC^{r'}}\widehat{\omega}\Id\omega^{\wedge d}=\frac{1}{2}\tr_{\IC^{r'}}\left(\widehat{\omega}\Id\omega+\Id\omega\widehat{\omega}\right)\wedge\Id\omega^{\wedge\left(d-1\right)}=\tr_{\IC^{r'}}\eta\wedge\Id\omega^{\wedge\left(d-1\right)}.
	\end{align}
On the other hand we have
\begin{align}
	\tr_{\IC^{r'}}\eta\wedge\Id\omega^{\wedge\left(d-1\right)}=\left(-1\right)^{\frac{d-1}{2}}\eta\wedge\sum_{\alpha\in\left\{1,\ldots,d\right\}^{d-1}}\tr_{\IC^{r'}}\left(\sigma^{\alpha_{1}}\cdot\ldots\cdot \sigma^{\alpha_{d-1}}\right)\Id F_{\alpha_{1}}\wedge\ldots\wedge\Id F_{\alpha_{d-1}}=0,
\end{align}
	because the trace on the right hand side vanishes if the $\alpha_{j}$ are pairwise different, while the exterior product vanishes if there are $k,l\in\left\{1,\ldots,d-1\right\}$ with $\alpha_{k}=\alpha_{l}$. We calculate
	\begin{align}
		\tr_{\IC^{r'}}\Id\omega^{\wedge d}&=\Ii^{d}\kappa_{\sigma}\sum_{\alpha,\beta\in\mathfrak{S}_{d}}\left(-1\right)^{\left|\alpha\right|+\left|\beta\right|}\p_{x^{\alpha_{1}}}F_{\beta_{1}}\cdot\ldots\cdot\p_{x^{\alpha_{d}}}F_{\beta_{d}}\ \Id x\nonumber\\
		&=\Ii^{d}d!\kappa_{\sigma}\det\mathrm{D}F\ \Id x,
	\end{align}
	where $\mathrm{D}F$ is the Jacobian of $F$. Finally we determine
	\begin{align}
		tr_{\IC^{r'}}\widehat{\omega}\eta\wedge\Id\omega^{\wedge\left(d-1\right)}&=\Ii^{d}\kappa_{\sigma}\sum_{\alpha,\beta\in\mathfrak{S}_{d}}\left(-1\right)^{\left|\alpha\right|+\left|\beta\right|}\frac{F_{\beta_{1}}F_{\alpha_{1}}}{\left|F\right|^2}\p_{\alpha_{1}}F_{\beta_{1}}\cdot\ldots\cdot\p_{\alpha_{d}}F_{\beta_{d}}\ \Id x\nonumber\\
		&=\Ii^{d}\left(d-1\right)!\kappa_{\sigma}\det\mathrm{D}F\ \Id x.
	\end{align}
		We apply $\tr_{\IC^{r'}}$ to (\ref{lemexeq1}), and obtain, expanding $A$ and $B$,
		\begin{align}
			tr_{\IC^{r'}}\left(\left(U^{V}\right)^{-1}\Id U^{V}\right)^{\wedge d}=&-2^{-\frac{d-1}{2}}\Ii^{d-1}\left(d-1\right)!\kappa_{\sigma}\left|F\right|^{-d}\left(\cos\left(2\left|F\right|\right)-1\right)^{\frac{d-1}{2}}\nonumber\\
			&\left(d-1+\left|F\right|\right)\det\mathrm{D}F\ \Id x.
	\end{align}
	Thus, with the choices $r=r'=2^{\frac{d-1}{2}}$, we apply the index formula of Theorem \ref{mainthm},
	\begin{align}
		\ind_{W}D_{V}=\left(2\pi\right)^{-\frac{d+1}{2}}\frac{\left(\frac{d-1}{2}\right)!}{d}\int_{\IR^{d}}\frac{\left(\left|F\right|+d-1\right)\left(\cos\left(2\left|F\right|\right)-1\right)^{\frac{d-1}{2}}}{\left|F\right|^{d}}\det\mathrm{D}F\ \Id x.
	\end{align}
	Note that the integrand extends continuously to $F=0$.
	\end{proof}
	
	The transformation theorem directly implies the following corollaries.
	
	\begin{Corollary}
		Assume additionally to Definition \ref{potentialexdef} that $F:\IR^{d}\rightarrow\Omega\subseteq\IR^{d}$ is a orientation preserving $C^{\infty}$-diffeomorphism. Then
		\begin{align}
			\ind_{W}D_{V}=\left(2\pi\right)^{-\frac{d+1}{2}}\frac{\left(\frac{d-1}{2}\right)!}{d}\int_{\Omega}\frac{\left(\left|y\right|+d-1\right)\left(\cos\left(2\left|y\right|\right)-1\right)^{\frac{d-1}{2}}}{\left|y\right|^{d}}\Id y.
		\end{align}		
	\end{Corollary}
	
	\begin{Corollary}
		Let $\Phi:\IR^{d}\rightarrow\IR^{d}$ be a orientation preserving $C^{\infty}$-diffeomorphism. Then for $V_{\Phi}\left(x,y\right):=V\left(\Phi\left(x\right),y\right)$, we have
		\begin{align}
			\ind_{W}D_{V_{\Phi}}=\ind_{W}D_{V}.
		\end{align}
	\end{Corollary}
	
	\begin{Corollary}\label{finalcor}
		For $s>1$, let $F_{C}\left(x\right):=C\int_{0}^{\left|x\right|}\langle r\rangle^{-s}\Id r\ \frac{x}{\left|x\right|}$, and let $V_{C}\left(x,y\right)=\Ii\phi\left(y\right)\sum_{j=1}\sigma^{j}F_{j}\left(x\right)$ for some $\phi\in C^{\infty}_{c}\left(\IR\right)$ with $\int_{\IR}\phi\left(y\right)\Id y=1$. Then the map,
		\begin{align}
			\IR\ni C\mapsto\ind_{W}D_{V_{C}}\in\IR,
		\end{align}
		is surjective.
	\end{Corollary}
	
	\begin{proof}
		For $C=0$ we get trivially $\ind_{W}D_{V_{C}}=0$. For $C\neq 0$ the function $F_{C}$ is a $C^{\infty}$-diffeomorphism of $\IR^{d}$ onto the ball of radius $r_{C}:=C\int_{0}^{\infty}\langle r\rangle^{-s}\Id r$ around $0$, which is checked elementarily. Thus
		\begin{align}
			\mathrm{sgn}\left(\det\mathrm{D}F_{C}\left(x\right)\right)=\mathrm{sgn}\left(\det\mathrm{D}F_{C}\left(0\right)\right)=\mathrm{sgn}\left(C^{d}\right)=\mathrm{sgn}\left(C\right).
		\end{align}
		By Lemma \ref{calclem}, and the transformation theorem we have
		\begin{align}\label{finalcoreq2}
			\ind_{W}D_{V_{C}}&=\left(2\pi\right)^{-\frac{d+1}{2}}\frac{\left(\frac{d-1}{2}\right)!}{d}\mathrm{sgn}\left(C\right)\int_{B_{r_{C}}\left(0\right)}\frac{\left(\left|y\right|+d-1\right)\left(\cos\left(2\left|y\right|\right)-1\right)^{\frac{d-1}{2}}}{\left|y\right|^{d}}\Id y\nonumber\\
			&=\frac{2^{\frac{d-1}{2}}}{d!}\mathrm{sgn}\left(C\right)\int_{0}^{r_{C}}\frac{r+d-1}{r}\left(\cos\left(2r\right)-1\right)^{\frac{d-1}{2}}\Id r,\ C\neq 0.
		\end{align}
		The claimed surjectivity follows by applying the estimate $\cos\left(2r\right)\leq 0$ if $r\in\left[\frac{\pi}{4},\frac{3\pi}{4}\right]+\pi\IZ$ to (\ref{finalcoreq2}), and the fact that the integrand in (\ref{finalcoreq2}) is non-negative or non-positive dependent only on the parity of $\frac{d-1}{2}$.
		\end{proof}
		
	\section{Acknowledgements}
	
	Discussing possible applications of \cite{F2} with Marcus Waurick, the author was made aware of the surjectivity problem for the Witten index on differential operators in dimensions higher than $2$, which gave the reason to construct the presented examples.


\begin{thebibliography}{99}
		
		\bibitem{BolGesGroSchSim} D. Boll\'e, F. Gesztesy, H. Grosse, W. Schweiger, B. Simon, \textit{Witten Index, Axial Anomaly, and
			Krein's Spectral Shift Function in Supersymmetric Quantum Mechanics}, J. Math. Phys. \textbf{28}
		(1987), no. \textbf{7}, p.1512-1525
		
		\bibitem{Cal} C. Callias, \textit{Axial Anomalies and Index Theorems on Open Spaces}, Commun. Math.
		Phys. \textbf{62}, 213-234, (1978).
		
		\bibitem{CGGLPSZ} A. Carey, F. Gesztesy, H. Grosse, G. Levitina, D. Potapov, F. Sukochev, D. Zanin, \textit{Trace Formulas for a Class of non-Fredholm Operators: A Review}, Reviews in Mathematical Physics \textbf{28}, No. 10 (2016)
		
		\bibitem{CGLPSZ} A. Carey, F. Gesztesy, G. Levitina, D. Potapov, F. Sukochev, D. Zanin, \textit{On Index Theory for Non-Fredholm Operators: A (1+1)-Dimensional Example}, arXiv:1509.01356 (2015)
		
		\bibitem{CGLS} A. Carey, F. Gesztesy, G. Levitina, F. Sukochev , \textit{The Spectral Shift Function and the Witten Index}, Spectral Theory and Mathematical Physics. Operator Theory: Advances and Applications, vol. \textbf{254} (2016), Birkh\"auser
		
		\bibitem{F2} O. F\"urst, \textit{Trace and Index of Dirac-Schr\"odinger Operators on Open Space with Operator Potentials}, arXiv:2311.02593 (2023)
		
		\bibitem{GesSim} F. Gesztesy, B. Simon, \textit{Topological Invariance of the Witten Index}, Journal of Functional Analysis \textbf{79} (1988), p.91-102
		
		\bibitem{Paz} A. Pazy, \textit{Semigroups of Linear Operators and Applications to Partial Differential Equations}, Applied Mathematical Sciences, vol. \textbf{44}, Springer (1983).
		
		\bibitem{Simon} B. Simon, \textit{Trace Ideals and Their Applications}, Mathematical Surveys and Monographs, Vol. \textbf{120}, Heidelberg: Am. Math. Soc. (2005).

	\end{thebibliography}
\end{document}